\documentclass{article}
\usepackage{amsmath}
\usepackage{amsfonts}
\usepackage{sw20bams}
\usepackage{everypage}
\usepackage{draftwatermark}

\newtheorem{theorem}{Theorem}

\newtheorem{conclusion}[theorem]{Conclusion}

\newtheorem{definition}[theorem]{Definition}
\newtheorem{example}[theorem]{Example}

\newtheorem{remark}[theorem]{Remark}

\input{tcilatex}
\SetWatermarkText{Accepted for publication in New Math and Nat. Comput.}

\begin{document}

\title{Mappings On Soft Classes \bigskip }
\author{Athar Kharal$^{1}$\thanks{%
Corresponding author. Ph. 0092 333 6261309} \thanks{%
This paper was originally submitted on 17 Oct., 2008 to Information Sciences
with the title of \textquotedblleft Mappings on Soft Sets\textquotedblright\
and was given MS\# INS-D-08-1231 by EES.} and B. Ahmad$^{2}$\thanks{%
Present Address: Department of Mathematics, King Abdul Aziz University, P.O.
Box 80203, Jeddah-21589, SAUDI ARABIA}\medskip  \\
$^{1}$National University of Sciences and Technology (NUST), \\
Islamabad, PAKISTAN\\
$^{2}$Centre for Advanced Studies in Pure and Applied Mathematics,\\
Bahauddin Zakariya University, Multan, PAKISTAN\\
Email: \textit{atharkharal@gmail.com, drbashir9@gmail.com}}
\date{}
\maketitle

\begin{abstract}
In this paper, we define the notion of a mapping on soft classes and study
several properties of images and inverse images of soft sets supported by
examples and counterexamples. Finally, these notions have been applied to
the problem of medical diagnosis in medical expert systems.\bigskip \newline
\textbf{Keywords :} Soft set; Soft class; Mapping on soft classes; Image of
soft set; Inverse image of a soft set; Medical diagnosis in medical expert
systems.\textit{\ \medskip }\newline
\end{abstract}

\section{\textbf{Introduction\protect\bigskip }}

To solve complicated problems in economics, engineering and environment, we
cannot successfully use classical methods because of different kinds of
incomplete knowledge, typical for those problems. There are four theories:
Theory of Probablity, Fuzzy Set Theory (FST) \cite{fst65zad}, Interval
Mathematics and Rough Set Theory (RST) \cite{rst83paw}, which we can
consider as mathematical tools for dealing with imperfect knowledge. All
these tools require the pre-specification of some parameter to start with,
e.g. probablity density function in Probablity Theory, membership function
in FST and an equivalence relation in RST. Such a requirement, seen in the
backdrop of imperfect or incomplete knowledge, raises many problems. At the
same time, incomplete knowledge remains the most glaring characteristic of
humanistic systems -- systems exemplified by economic systems, biological
systems, social systems, political systems, information systems and, more
generally, man-machine systems of various types.

Noting problems in parameter specification Molodtsov \cite{sfst99mol}
introduced the notion of soft set to deal with problems of incomplete
information. Soft Set Theory (SST) does not require the specification of a
parameter, instead it accommodates approximate descriptions of an object as
its starting point. This makes SST a natural mathematical formalism for
approximate reasoning. We can use any parametrization we prefer: with the
help of words, sentences, real numbers, functions, mappings, and so on. This
means that the problem of setting the membership function or any similar
problem does not arise in SST.

SST has seminal links with rough set technique of automated knowledge
discovery. Soft set being collection of information granules, bears a close
resemblance with rough sets. A rough set \cite{rst83paw} is defined to be a
set given by an upper and a lower approximation sets from a universe of
information granules. Aktas and Cagman \cite{sfal07akt} have shown that,
both, an arbitrary rough set or an arbitrary fuzzy set may be expressed as a
soft set. Hence Soft Set Theory is more general a set up than RST and/or
FST. Links between soft sets and information systems and hence to Rough Set
Theory, have been further studied in \cite{sfst05pei,sfst05xia,sfst08zou}.
On the other hand, techniques from RST have been found applicable to SST,
due to the affinity of both approaches. Maji, Biswas and Roy \cite{sfst02maj}
applied the technique of knowledge reduction to the information table
induced by a soft set. Another parametrization reduction of soft set was
proposed in \cite{sfst03che,sfst05che}. Recently Z. Kong \textit{et.al.} has
also proposed yet another novel method of parameter reduction in \cite%
{sfst08konB}.

Applications of Soft Set Theory in other disciplines and real life problems
are now catching momentum. Molodtsov \cite{sfst99mol} successfully applied
the Soft Set Theory into several directions, such as smoothness of
functions, Riemann-integration, Perron integration, Theory of Probability,
Theory of Measurement and so on. Kovkov \textit{et.al.} \cite{sfst07kov} has
found promising results by applying soft sets to Optimization Theory, Game
Theory and Operations Research. Maji and Roy \cite{sfst02maj} applied soft
sets in a multicriteria decision making (MCDM) problem. It is based on the
notion of knowledge reduction of rough sets. Mushrif and Sengupta \cite%
{sfst06mus} based their algorithm for natural texture classification on soft
sets. This algorithm has a low computational complexity when compared to a
Bayes technique based method for texture classification. Zou and Xia \cite%
{sfst08zou} have exploited the link between soft sets and data analysis in
incomplete information systems.

In this paper, we first introduce the notion of mapping on soft classes.
Soft classes are collections of soft sets (Definition \ref{df_sftClass}). We
also define and study the properties of soft images and soft inverse images
of soft sets, and support them with examples and counterexamples. Finally,
these notions have been applied to the problem of medical diagnosis in
medical expert systems.\bigskip \newline

\section{\textbf{Preliminaries\protect\bigskip }}

First we recall basic definitions and results.

\begin{definition}
\cite{sfst99mol}\label{st-SoftSet} A pair $(F,A)$ is called a soft set over $%
X$, where $F$ is a mapping given by $F:A\rightarrow P(X).$\newline
In other words, a soft set over $X$ is a parametrized family of subsets of
the universe $X.$ For $\varepsilon \in A,$ $F(\varepsilon )$ may be
considered as the set of $\varepsilon $-approximate elements of the soft set 
$(F,A)$. Clearly a soft set is not a set in ordinary sense.
\end{definition}

\begin{definition}
\cite{sfst05pei}\label{st-subset} For two soft sets $(F,A)$ and $(G,B)$ over 
$X$, we say that $(F,A)$ is a soft subset of $(G,B),$ if\newline
$(i)$ $A\subseteq B,$ and\newline
$(ii)$ $\forall ~\varepsilon \in A,F(\varepsilon )\subseteq G(\varepsilon )$.%
\newline
We write $(F,A)~\widetilde{\subseteq }~(G,B)$. $(F,A)$ is said to be a soft
super set of $(G,B)$, if $(G,B)$ is a soft subset of $(F,A)$. We denote it
by $(F,A)~\widetilde{\supseteq }~(G,B)$.
\end{definition}

\begin{definition}
\cite{sfst03maj}\label{st-Union} Union of two soft sets $(F,A)$ and $(G,B)$
over the common universe $X$ is the soft set $(H,C),$ where $C=A\cup B,$ and 
$\forall ~\varepsilon \in C,$%
\begin{equation*}
H(\varepsilon )=\left\{ 
\begin{array}{ccc}
F(\varepsilon ), & if & \varepsilon \in A-B, \\ 
G(\varepsilon ), & if & \varepsilon \in B-A, \\ 
F(\varepsilon )\cup G(\varepsilon ), & if & \varepsilon \in A\cap B.%
\end{array}%
\right.
\end{equation*}%
We write $(F,A)~\tilde{\cup}~(G,B)=(H,C).$
\end{definition}

Maji, Biswas and Roy defined the intersection of two soft sets as:

\begin{definition}
\label{intersection_def_of_Maji}\cite{sfst03maj} Intersection of two soft
sets $(F,A)$ and $(G,B)$ over $X$ is a soft set $(H,C),$ where $C=A\cap B,$
and $\forall ~\varepsilon \in C,H(\varepsilon )=F(\varepsilon
)~or~G(\varepsilon )$, (as both are same set), and is written as $(F,A)%
\tilde{\cap}(G,B)=(H,C).$
\end{definition}

Pei and Miao pointed out that generally $F\left( \varepsilon \right) $ and $%
G\left( \varepsilon \right) $ may not be identical and thus revised the
above definition as:

\begin{definition}
\cite{sfst05pei}\label{intersection_def_of_Pei} Let $(F,A)$ and $(G,B)$ be
two soft sets over $X$. Intersection (also called bi-intersction by Feng
et.al. \cite{sfalg08fen}) of two soft sets $(F,A)$ and $(G,B)$ is a soft set 
$(H,C),$ where $C=A\cap B,$ and $\forall ~\varepsilon \in C,H(\varepsilon
)=F(\varepsilon )\cap G(\varepsilon )$. We write $(F,A)\tilde{\cap}%
(G,B)=(H,C).$
\end{definition}

We further point out that in Definition \ref{intersection_def_of_Pei}, $%
A\cap B$ must be nonempty to avoid the degenerate case. Hence the definition %
\ref{intersection_def_of_Pei} is improved as:

\begin{definition}
\label{intersection_def_Ours}Let $(F,A)$ and $(G,B)$ be two soft sets over $%
X $ with $A\cap B\neq \phi $. Intersection of two soft sets $(F,A)$ and $%
(G,B)$ is a soft set $(H,C),$ where $C=A\cap B,$ and $\forall ~\varepsilon
\in C,H(\varepsilon )=F(\varepsilon )\cap G(\varepsilon )$. We write $(F,A)%
\tilde{\cap}(G,B)=(H,C).$
\end{definition}

\section{\textbf{Mappings on Soft Classes\protect\bigskip }}

First we define:

\begin{definition}
\label{df_sftClass}Let $X$ \ be a universe and $E$ a set of attributes. Then
the collection of all soft sets over $X$ with attributes from $E$ is called
a soft class and is denoted as $\left( X,E\right) $.
\end{definition}

\begin{definition}
Let $\left( X,E\right) $ and $\left( Y,E^{\prime }\right) $ be soft classes.
Let $u:X\rightarrow Y$ and $p:E\rightarrow E^{\prime }$ be mappings. Then a
mapping $f:\left( X,E\right) \rightarrow \left( Y,E^{\prime }\right) $ is
defined as: for a soft set $\left( F,A\right) $ in $\left( X,E\right) $, $%
\left( f\left( F,A\right) ,B\right) ,$ $B=p\left( A\right) \subseteq
E^{\prime }$ is a soft set in $\left( Y,E^{\prime }\right) $ given by%
\begin{equation*}
f\left( F,A\right) \left( \beta \right) =\left\{ 
\begin{array}{cc}
u\left( \dbigcup\limits_{\alpha \in p^{-1}\left( \beta \right) \cap
A}F\left( \alpha \right) \right) , & \text{if }p^{-1}\left( \beta \right)
\cap A\neq \phi , \\ 
&  \\ 
\phi , & \text{otherwise,}%
\end{array}%
\right.
\end{equation*}%
for $\beta \in B\subseteq E^{\prime }.$ $\left( f\left( F,A\right) ,B\right) 
$ is called a soft image of a soft set $\left( F,A\right) .$ If $B=E^{\prime
},$ then we shall write $f\left( \left( F,A\right) ,E^{\prime }\right) $ as $%
f\left( F,A\right) .$
\end{definition}

\begin{definition}
Let $f:\left( X,E\right) \rightarrow \left( Y,E^{\prime }\right) $ be a
mapping from a soft class $\left( X,E\right) $ to another soft class $\left(
Y,E^{\prime }\right) ,$ and $\left( G,C\right) ,$ a soft set in soft class $%
\left( Y,E^{\prime }\right) ,$ where $C\subseteq E^{\prime }.$ Let $%
u:X\rightarrow Y$ and $p:E\rightarrow E^{\prime }$ be mappings. Then $\left(
f^{-1}\left( G,C\right) ,D\right) ,~D=p^{-1}\left( C\right) ,$ is a soft set
in the soft class $\left( X,E\right) ,$ defined as: 
\begin{equation*}
f^{-1}\left( G,C\right) \left( \alpha \right) =\left\{ 
\begin{array}{cc}
u^{-1}\left( G\left( p\left( \alpha \right) \right) \right) ,~ & p\left(
\alpha \right) \in C, \\ 
&  \\ 
\phi , & otherwise%
\end{array}%
\right.
\end{equation*}%
for $\alpha \in D\subseteq E.$ $\left( f^{-1}\left( G,C\right) ,D\right) $
is called a soft inverse image of $\left( G,C\right) .$ Hereafter we shall
write $\left( f^{-1}\left( G,C\right) ,E\right) $ as $f^{-1}\left(
G,C\right) .$
\end{definition}

\begin{example}
\label{ex_illustrating_def_of_fn}Let $X=\left\{ a,b,c\right\} ,~Y=\left\{
x,y,z\right\} ,~E=\left\{ e_{1},e_{2},e_{3},e_{4}\right\} ,$ $E^{\prime
}=\left\{ e_{1}^{\prime },e_{2}^{\prime },e_{3}^{\prime }\right\} $ and $%
\left( X,E\right) ,$ $\left( Y,E^{\prime }\right) ,$ soft classes. Define $%
u:X\rightarrow Y$ and $p:E\rightarrow E^{\prime }$ as: 
\begin{eqnarray*}
u\left( a\right) &=&y,~u\left( b\right) =z,~u\left( c\right) =y, \\
p\left( e_{1}\right) &=&e_{3}^{\prime },~p\left( e_{2}\right) =e_{3}^{\prime
},~p\left( e_{3}\right) =e_{2}^{\prime },~p\left( e_{4}\right)
=e_{3}^{\prime }.
\end{eqnarray*}%
Choose two soft sets over $X$ and $Y$ respectively as:%
\begin{eqnarray*}
\left( F,A\right) &=&\left\{ e_{2}=\left\{ {}\right\} ,~e_{3}=\left\{
a\right\} ,~e_{4}=\left\{ a,b,c\right\} \right\} , \\
\left( G,C\right) &=&\left\{ e_{1}^{\prime }=\left\{ x,z\right\}
,~e_{2}^{\prime }=\left\{ y\right\} \right\} .
\end{eqnarray*}%
Then the mapping $f:\left( X,E\right) \rightarrow \left( Y,E^{\prime
}\right) $ is given as: for a soft set $\left( F,A\right) $ in $\left(
X,E\right) ,$ $\left( f\left( F,A\right) ,B\right) $ where $B=p\left(
A\right) =\left\{ e_{2}^{\prime },e_{3}^{\prime }\right\} $ is a soft set in 
$\left( Y,E^{\prime }\right) $ obtained as follows:%
\begin{eqnarray*}
f\left( F,A\right) e_{2}^{\prime } &=&u\left( \bigcup F\left( \left\{
e_{3}\right\} \right) \right) =u\left( \left\{ a\right\} \right) =\left\{
y\right\} ,~~~\text{(since }p^{-1}\left( e_{2}^{\prime }\right) \cap
A=\left\{ e_{3}\right\} \text{),} \\
f\left( F,A\right) e_{3}^{\prime } &=&u\left( \bigcup_{\alpha \in
p^{-1}\left( e_{3}^{\prime }\right) \cap A}F\left( \alpha \right) \right)
=u\left( \left\{ F\left( e_{2}\right) \cup F\left( e_{4}\right) \right\}
\right) \\
&=&u\left( \left\{ {}\right\} \cup \left\{ a,b,c\right\} \right) =u\left(
\left\{ a,b,c\right\} \right) =\left\{ y,z\right\} ,\text{ (since }%
p^{-1}\left( e_{3}^{\prime }\right) \cap A=\left\{ e_{2},e_{4}\right\} \text{%
).}
\end{eqnarray*}%
Hence%
\begin{equation*}
\left( f\left( F,A\right) ,B\right) =\left\{ ~e_{2}^{\prime }=\left\{
y\right\} ,~e_{3}^{\prime }=\left\{ y,z\right\} \right\} .
\end{equation*}%
Next for the soft images, we have 
\begin{equation*}
f^{-1}\left( G,C\right) e_{3}=u^{-1}\left( G\left( p\left( e_{3}\right)
\right) \right) =u^{-1}\left( G\left( e_{2}^{\prime }\right) \right)
=u^{-1}\left( \left\{ y\right\} \right) =\left\{ a,c\right\} ,
\end{equation*}%
where $D=p^{-1}\left( C\right) =\left\{ e_{3}\right\} .$

Hence, we have $\left( f^{-1}\left( G,C\right) ,D\right) =\left\{
e_{3}=\left\{ a,c\right\} \right\} .$
\end{example}

\begin{definition}
\label{df_union_of_images}Let $f$ $:\left( X,E\right) \rightarrow \left(
Y,E^{\prime }\right) $ be a mapping and $\left( F,A\right) $, $\left(
G,B\right) $ soft sets in $\left( X,E\right) $. Then for $\beta \in
E^{\prime },$ soft union and intersection of soft images of (F, A) and (G,
B) in (X, E) are defined as : 
\begin{eqnarray*}
\left( f\left( F,A\right) ~\widetilde{\cup }~f\left( G,B\right) \right)
\beta &=&f\left( F,A\right) \beta \cup f\left( G,B\right) \beta , \\
\left( f\left( F,A\right) ~\widetilde{\cap }~f\left( G,B\right) \right)
\beta &=&f\left( F,A\right) \beta \cap f\left( G,B\right) \beta .
\end{eqnarray*}
\end{definition}

\begin{definition}
\label{df_union_of_inv_images}Let $f$ $:\left( X,E\right) \rightarrow \left(
Y,E^{\prime }\right) $ be a mapping and $\left( F,A\right) $, $\left(
G,B\right) $ soft sets in.$\left( Y,E^{\prime }\right) .$ Then for $\alpha
\in E,$ soft union and intersection of soft inverse images of soft sets $%
\left( F,A\right) ,\left( G,B\right) $\ are defined as:%
\begin{eqnarray*}
\left( f^{-1}\left( F,A\right) ~\widetilde{\cup }~f^{-1}\left( G,B\right)
\right) \alpha &=&f^{-1}\left( F,A\right) \alpha \cup f^{-1}\left(
G,B\right) \alpha , \\
\left( f^{-1}\left( F,A\right) ~\widetilde{\cap }~f^{-1}\left( G,B\right)
\right) \alpha &=&f^{-1}\left( F,A\right) \alpha \cap f^{-1}\left(
G,B\right) \alpha .
\end{eqnarray*}
\end{definition}

\begin{remark}
Note that the null (resp. absolute) soft set as defined by Maji \textit{%
et.al. \cite{sfst03maj}, is not unique in a soft space }$\left( X,E\right) ,$%
\textit{\ rather it depends upon }$A\subseteq E.$ Therefore, we denote it by 
$\widetilde{\Phi }_{A}$ (resp. $\widetilde{X}_{A}$)$.$ If $A=E,$ then we
denote it simply by $\widetilde{\Phi }$ (resp. $\widetilde{X}$)$,$ which is
unique null (resp. absolute) soft set, called full null (resp. full
absolute) soft set.
\end{remark}

\begin{theorem}
\label{th_fn_properties}Let $f$ $:\left( X,E\right) \rightarrow \left(
Y,E^{\prime }\right) $, $u$ $:X\rightarrow Y$ and $p:E\rightarrow E^{\prime
} $ be mappings. Then for soft sets $\left( F,A\right) ,$ $\left( G,B\right) 
$ and a family of soft sets $\left( F_{i},A_{i}\right) $ in the soft class $%
\left( X,E\right) ,$ we have:\newline
$\left( 1\right) $ $f\left( \widetilde{\Phi }\right) =\widetilde{\Phi }.$%
\newline
$\left( 2\right) $ $f\left( \widetilde{X}\right) ~\widetilde{\subseteq }~%
\widetilde{Y}.$ \newline
$\left( 3\right) ~f\left( \left( F,A\right) \widetilde{\cup }\left(
G,B\right) \right) =f\left( F,A\right) \widetilde{\cup }f\left( G,B\right) .$
\newline
\qquad In general $f\left( \underset{i}{\widetilde{\cup }}\left(
F_{i},A_{i}\right) \right) =\underset{i}{\widetilde{\cup }}f\left(
F_{i},A_{i}\right) .$\newline
$\left( 4\right) $ $f\left( \left( F,A\right) \widetilde{\cap }\left(
G,B\right) \right) ~\widetilde{\supseteq }~f\left( F,A\right) \widetilde{%
\cap }f\left( G,B\right) ,$ \newline
\qquad In general $f\left( \underset{i}{\widetilde{\cap }}\left(
F_{i},A_{i}\right) \right) \widetilde{\subseteq }\underset{i}{\widetilde{%
\cap }}f\left( F_{i},A_{i}\right) .$\newline
$\left( 5\right) $ If $\left( F,A\right) ~\widetilde{\subseteq }~\left(
G,B\right) ,$ then $f\left( F,A\right) ~\widetilde{\subseteq }~f\left(
G,B\right) .$
\end{theorem}

\begin{proof}
We only prove $\left( 3\right) -\left( 5\right) .$\newline
$\left( 3\right) $ For $\beta \in E^{\prime },$ we show that $f\left( \left(
F,A\right) \widetilde{\cup }\left( G,B\right) \right) \beta =\left( f\left(
F,A\right) ~\widetilde{\cup }~f\left( G,B\right) \right) \beta .$ Consider%
\begin{eqnarray*}
f\left( \left( F,A\right) \widetilde{\cup }\left( G,B\right) \right) \beta
&=&f\left( H,A\cup B\right) \beta \text{ ~~~(say)} \\
&=&\left\{ 
\begin{array}{cl}
u\left( \dbigcup\limits_{\alpha \in p^{-1}\left( \beta \right) \cap \left(
A\cup B\right) }H\left( \alpha \right) \right) , & \text{if }p^{-1}\left(
\beta \right) \cap \left( A\cup B\right) \neq \phi , \\ 
\phi , & \text{otherwise,}%
\end{array}%
\right. \\
\text{where }H\left( \alpha \right) &=&\left\{ 
\begin{array}{cc}
F\left( \alpha \right) , & \alpha \in A-B \\ 
G\left( \alpha \right) , & \alpha \in B-A \\ 
F\left( \alpha \right) \cup G\left( \alpha \right) , & \alpha \in A\cap B.%
\end{array}%
\right.
\end{eqnarray*}%
We consider the case, when $p^{-1}\left( \beta \right) \cap \left( A\cup
B\right) \neq \phi ,$ as otherwise it is trivial. Then%
\begin{equation}
f\left( \left( F,A\right) \widetilde{\cup }\left( G,B\right) \right) \beta
=u\left( \dbigcup \left\{ 
\begin{array}{cc}
F\left( \alpha \right) , & \alpha \in \left( A-B\right) \cap p^{-1}\left(
\beta \right) \\ 
G\left( \alpha \right) , & \alpha \in \left( B-A\right) \cap p^{-1}\left(
\beta \right) \\ 
F\left( \alpha \right) \cup G\left( \alpha \right) , & \alpha \in \left(
A\cap B\right) \cap p^{-1}\left( \beta \right)%
\end{array}%
\right. \right)  \tag{I}  \label{eq3-1}
\end{equation}%
Next, for the non-trivial case, using Definition \ref{df_union_of_images}
and for $\beta \in E^{\prime }$, we have%
\begin{eqnarray}
\left( f\left( F,A\right) ~\widetilde{\cup }~f\left( G,B\right) \right)
\beta &=&f\left( F,A\right) \beta ~\cup ~f\left( G,B\right) \beta  \notag \\
&=&u\left( \dbigcup\limits_{\alpha \in p^{-1}\left( \beta \right) \cap
A}F\left( \alpha \right) \right) \cup u\left( \dbigcup\limits_{\alpha \in
p^{-1}\left( \beta \right) \cap B}G\left( \alpha \right) \right)  \notag \\
&=&u\left( \dbigcup\limits_{\alpha \in p^{-1}\left( \beta \right) \cap
A}F\left( \alpha \right) \cup \dbigcup\limits_{\alpha \in p^{-1}\left( \beta
\right) \cap B}G\left( \alpha \right) \right)  \notag \\
&=&u\left( \dbigcup \left\{ 
\begin{array}{cc}
F\left( \alpha \right) , & \alpha \in \left( A-B\right) \cap p^{-1}\left(
\beta \right) \\ 
G\left( \alpha \right) , & \alpha \in \left( B-A\right) \cap p^{-1}\left(
\beta \right) \\ 
F\left( \alpha \right) \cup G\left( \alpha \right) , & \alpha \in \left(
A\cap B\right) \cap p^{-1}\left( \beta \right) .%
\end{array}%
\right. \right)  \TCItag{II}  \label{eq3-2}
\end{eqnarray}%
\newline
From (\ref{eq3-1}) and (\ref{eq3-2}), we have $\left( 3\right) .\bigskip
\medskip $\newline
$\left( 4\right) $ For $\beta \in E^{\prime },$ we show that $f\left( \left(
F,A\right) \widetilde{\cap }\left( G,B\right) \right) \beta \subseteq \left(
f\left( F,A\right) ~\widetilde{\cap }~f\left( G,B\right) \right) \beta .$
Consider%
\begin{equation*}
f\left( \left( F,A\right) \widetilde{\cap }\left( G,B\right) \right) \beta
=f\left( H,A\cap B\right) \beta =\left\{ 
\begin{array}{cl}
u\left( \dbigcup\limits_{\alpha \in p^{-1}\left( \beta \right) \cap \left(
A\cap B\right) }H\left( \alpha \right) \right) , & \text{if }p^{-1}\left(
\beta \right) \cap \left( A\cap B\right) \neq \phi , \\ 
&  \\ 
\phi , & \text{otherwise,}%
\end{array}%
\right.
\end{equation*}%
where $H\left( \alpha \right) =F\left( \alpha \right) \cap G\left( \alpha
\right) .$ We consider the case when $p^{-1}\left( \beta \right) \cap \left(
A\cap B\right) \neq \phi ,$ as otherwise it is trivial. Thus%
\begin{equation*}
f\left( H,A\cap B\right) \beta =u\left( \dbigcup\limits_{\alpha \in
p^{-1}\left( \beta \right) \cap \left( A\cap B\right) }H\left( \alpha
\right) \right) =u\left( \dbigcup\limits_{\alpha \in p^{-1}\left( \beta
\right) \cap \left( A\cap B\right) }\left( F\left( \alpha \right) \cap
G\left( \alpha \right) \right) \right) ,
\end{equation*}%
\newline
or 
\begin{equation}
f\left( \left( F,A\right) \widetilde{\cap }\left( G,B\right) \right) \beta
=u\left( \dbigcup\limits_{\alpha \in p^{-1}\left( \beta \right) \cap \left(
A\cap B\right) }\left( F\left( \alpha \right) \cap G\left( \alpha \right)
\right) \right)  \tag{I}  \label{eq4-1}
\end{equation}%
On the other hand, using Definition \ref{df_union_of_images}, we have%
\begin{eqnarray*}
\left( f\left( \left( F,A\right) \right) ~\widetilde{\cap }~f\left( \left(
G,B\right) \right) \right) \beta &=&f\left( F,A\right) \beta ~\cap ~f\left(
G,B\right) \beta \\
&=&\left( \left\{ 
\begin{array}{cl}
u\left( \dbigcup\limits_{\alpha \in p^{-1}\left( \beta \right) \cap
A}F\left( \alpha \right) \right) , & \text{if }p^{-1}\left( \beta \right)
\cap A\neq \phi , \\ 
&  \\ 
\phi , & \text{otherwise.}%
\end{array}%
\right. \right) ~\cap \\
&&\left( \left\{ 
\begin{array}{cl}
u\left( \dbigcup\limits_{\alpha \in p^{-1}\left( \beta \right) \cap
B}G\left( \alpha \right) \right) , & \text{if }p^{-1}\left( \beta \right)
\cap B\neq \phi , \\ 
&  \\ 
\phi , & \text{otherwise.}%
\end{array}%
\right. \right)
\end{eqnarray*}%
Ignoring the trivial case, we get%
\begin{multline*}
\left( f\left( F,A\right) \widetilde{\cap }f\left( G,B\right) \newline
\right) \beta =u\left( \dbigcup\limits_{\alpha \in p^{-1}\left( \beta
\right) \cap A}F\left( \alpha \right) \right) ~\cap ~u\left(
\dbigcup\limits_{\alpha \in p^{-1}\left( \beta \right) \cap B}G\left( \alpha
\right) \right) \\[1pt]
\supseteq u\left( \dbigcup\limits_{\alpha \in p^{-1}\left( \beta \right)
\cap \left( A\cap B\right) }\left( F\left( \alpha \right) \cap G\left(
\alpha \right) \right) \right) =f\left( \left( F,A\right) ~\widetilde{\cap }%
~\left( G,B\right) \right) \beta \text{ }
\end{multline*}%
$\bigskip $\newline
$\left( 5\right) ~$For $\beta \in E^{\prime }$%
\begin{equation*}
f\left( F,A\right) \beta =\left\{ 
\begin{array}{cc}
u\left( \dbigcup\limits_{\alpha \in p^{-1}\left( \beta \right) \cap
A}F\left( \alpha \right) \right) , & \text{if }p^{-1}\left( \beta \right)
\cap A\neq \phi , \\ 
\phi , & \text{otherwise.}%
\end{array}%
\right.
\end{equation*}%
We consider the case when $p^{-1}\left( \beta \right) \cap A\neq \phi ,$ as
otherwise it is trivial. Then%
\begin{equation*}
f\left( F,A\right) \beta =u\left( \dbigcup\limits_{\alpha \in p^{-1}\left(
\beta \right) \cap A}F\left( \alpha \right) \right) \subseteq u\left(
\dbigcup\limits_{\alpha \in p^{-1}\left( \beta \right) \cap B}G\left( \alpha
\right) \right) =f\left( G,B\right) \beta .
\end{equation*}%
This gives $\left( 5\right) $.\bigskip \newline
\end{proof}

In Theorem \ref{th_fn_properties}, inequalities $\left( 2\right) $ and $%
\left( 4\right) ,$ cannot be reversed, in general, as is shown in the
following:

\begin{example}
The soft classes $\left( X,E\right) ,\left( Y,E\right) $ and mapping $\ f$ $%
:\left( X,E\right) \rightarrow \left( Y,E^{\prime }\right) $ are as defined
in Example \ref{ex_illustrating_def_of_fn}. Then%
\begin{equation*}
\widetilde{Y}\not\subseteq f\left( \widetilde{X}\right) =\left\{
e_{1}^{^{\prime }}=\left\{ x,z\right\} ,e_{2}^{^{\prime }}=\left\{
x,z\right\} ,e_{3}^{^{\prime }}=\left\{ x,z\right\} \right\} .
\end{equation*}%
This shows that the reversal of inequality $\left( 2\right) $ is not true.%
\newline
To show that the reversal of $\left( 4\right) $ does not hold, choose soft
sets in $\left( X,E\right) $ as:%
\begin{eqnarray*}
\left( F,A\right) &=&\left\{ e_{1}=\left\{ c\right\} ,e_{2}=\left\{
b,c\right\} ,e_{3}=\left\{ a,b,c\right\} \right\} , \\
\left( G,B\right) &=&\left\{ e_{1}=\left\{ a\right\} ,e_{2}=\left\{
a,c\right\} ,e_{3}=\left\{ b\right\} ,e_{4}=\left\{ b,c\right\} \right\} .
\end{eqnarray*}%
Then calculations show that%
\begin{equation*}
f\left( F,A\right) \widetilde{\cap }f\left( G,B\right) =\left\{ \left\{
e_{2}^{\prime }=\left\{ z\right\} \right\} ,e_{3}^{\prime }=\left\{
y,z\right\} \right\} \widetilde{\not\subseteq }\left\{ \left\{ e_{2}^{\prime
}=\left\{ z\right\} ,e_{3}^{\prime }=\left\{ y\right\} \right\} \right\}
=f\left( \left( F,A\right) \widetilde{\cap }\left( G,B\right) \right) .
\end{equation*}
\end{example}

\begin{theorem}
\label{th_fn_inv_properties}Let $f$ $:\left( X,E\right) \rightarrow \left(
Y,E^{\prime }\right) $, $u$ $:X\rightarrow Y$ and $p:E\rightarrow E^{\prime
} $ be mappings. Then for soft sets $\left( F,A\right) ,$ $\left( G,B\right) 
$ and a family of soft sets $\left( F_{i},A_{i}\right) $ in the soft class $%
\left( Y,E^{\prime }\right) ,$ we have:\newline
$\left( 1\right) $ $f^{-1}\left( \widetilde{\Phi }\right) =\widetilde{\Phi }%
. $\newline
$\left( 2\right) $ $f^{-1}\left( \widetilde{Y}\right) =\widetilde{X}.$%
\newline
$\left( 3\right) $ $f^{-1}\left( \left( F,A\right) ~\widetilde{\cup }~\left(
G,B\right) \right) =f^{-1}\left( F,A\right) ~\widetilde{\cup }~f^{-1}\left(
G,B\right) .$ \newline
\qquad In general $f^{-1}\left( \underset{i}{\widetilde{\cup }}\left(
F_{i},A_{i}\right) \right) =\underset{i}{\widetilde{\cup }}~f^{-1}\left(
F_{i},A_{i}\right) .$\newline
$\left( 4\right) $ $f^{-1}\left( \left( F,A\right) ~\widetilde{\cap }~\left(
G,B\right) \right) =f^{-1}\left( F,A\right) ~\widetilde{\cap }~f^{-1}\left(
G,B\right) .$ \newline
\qquad In general $f^{-1}\left( \underset{i}{\widetilde{\cap }}\left(
F_{i},A_{i}\right) \right) =\underset{i}{\widetilde{\cap }}~f^{-1}\left(
F_{i},A_{i}\right) .$\newline
$\left( 5\right) $ If $\left( F,A\right) \widetilde{\subseteq }\left(
G,B\right) ,$ then $f^{-1}\left( F,A\right) ~\widetilde{\subseteq }%
~f^{-1}\left( G,B\right) .$\newline
\end{theorem}

\begin{proof}
We only prove $\left( 3\right) -\left( 5\right) .$\newline
$\left( 3\right) $ For $\alpha \in E$%
\begin{eqnarray}
f^{-1}\left( \left( F,A\right) \widetilde{\cup }\left( G,B\right) \right)
\alpha &=&f^{-1}\left( H,A\cup B\right) \alpha  \notag \\
&=&u^{-1}\left( H\left( p\left( \alpha \right) \right) \right) ,\text{ }%
p\left( \alpha \right) \in A\cup B  \notag \\
&=&u^{-1}\left( H\left( \beta \right) \right) ,\text{ where }\beta =p\left(
\alpha \right)  \notag \\
&=&u^{-1}\left( \left\{ 
\begin{array}{cc}
F\left( \beta \right) , & \beta \in A-B \\ 
G\left( \beta \right) , & \beta \in B-A \\ 
F\left( \beta \right) \cup G\left( \beta \right) , & \beta \in A\cap B%
\end{array}%
\right. \right)  \TCItag{I}  \label{finv-eq3-1}
\end{eqnarray}%
Next, using Definition \ref{df_union_of_inv_images}, we have%
\begin{eqnarray}
\left( f^{-1}\left( F,A\right) ~\widetilde{\cup }~f^{-1}\left( G,B\right)
\right) \alpha &=&f^{-1}\left( F,A\right) \alpha ~\cup ~f^{-1}\left(
G,B\right) \alpha  \notag \\
&=&u^{-1}\left( F\left( p\left( \alpha \right) \right) \right) ~\cup
~u^{-1}\left( G\left( p\left( \alpha \right) \right) \right) ,\text{ }%
p\left( \alpha \right) \in A\cap B  \notag \\
&=&u^{-1}\left( \left\{ 
\begin{array}{cc}
F\left( \beta \right) , & \beta \in A-B \\ 
G\left( \beta \right) , & \beta \in B-A \\ 
F\left( \beta \right) \cup G\left( \beta \right) , & \beta \in A\cap B%
\end{array}%
\right. \right) \text{ }  \TCItag{II}  \label{finv-eq3-2}
\end{eqnarray}%
where $\beta =p\left( \alpha \right) .$

From (\ref{finv-eq3-1}) and (\ref{finv-eq3-2}), we obtain $\left( 3\right) $%
.\bigskip \newline
$\left( 4\right) $ For $\alpha \in E$%
\begin{eqnarray*}
\left( f^{-1}\left( \left( F,A\right) ~\widetilde{\cap }~\left( G,B\right)
\right) \right) \alpha &=&\left( f^{-1}\left( H,A\cap B\right) \right) \alpha
\\
&=&u^{-1}\left( H\left( p\left( \alpha \right) \right) \right) ,\text{ }%
p\left( \alpha \right) \in A\cap B \\
&=&u^{-1}\left( F\left( \beta \right) \cap G\left( \beta \right) \right) ,%
\text{ where }\beta =p\left( \alpha \right) . \\
&=&u^{-1}\left( F\left( \beta \right) \right) \cap u^{-1}\left( G\left(
\beta \right) \right) \\
&=&u^{-1}\left( F\left( p\left( \alpha \right) \right) \right) \cap
u^{-1}\left( G\left( p\left( \alpha \right) \right) \right) \\
&=&\left( f^{-1}\left( F,A\right) \widetilde{\cap }f^{-1}\left( G,B\right)
\right) \alpha .
\end{eqnarray*}%
This proves $\left( 4\right) $.\bigskip \newline
$\left( 5\right) ~$For $\alpha \in E,$ consider%
\begin{eqnarray*}
f^{-1}\left( F,A\right) \alpha &=&u^{-1}\left( F\left( p\left( \alpha
\right) \right) \right) =u^{-1}\left( F\left( \beta \right) \right) ,\text{
where }\beta =p\left( \alpha \right) \\
&\subseteq &u^{-1}\left( G\left( \beta \right) \right) =u^{-1}\left( G\left(
p\left( \alpha \right) \right) \right) =f^{-1}\left( G,B\right) \alpha .
\end{eqnarray*}%
This gives $\left( 5\right) .\bigskip $
\end{proof}

\section{\textbf{An Application in Medical Expert Systems\protect\bigskip }}

An important task of a medical expert system is to transform a patient's
complaints/symptoms into a set of possible causes and their respective
importance from the view point of a medical specialist. A patient's case may
easily be encoded into a soft set. Suppose following is the narration by the
patient:

\begin{quotation}
I have three main complaints viz. burning in stomach, headache and
sleeplessness. Whatever sleep I get, is unrefreshing and semi-conscious. I
mean, I am always aware what is transpiring in the room when asleep. I have
some pain in joints and backbone. To a less degree I also suffer from
depression and anxiety.
\end{quotation}

This can be written as the following soft set:%
\begin{equation*}
\left( F,A\right) =\left\{ 
\begin{array}{c}
\text{high importance}=\left\{ \text{burning in stomach, headache,
sleeplessness}\right\} ,~ \\ 
\text{medium importance}=\left\{ \text{semi-coscious sleep}\right\} , \\ 
\text{low importance}=\left\{ \text{joint pain, backbone pain, depression,
anxiety}\right\}%
\end{array}%
\right\}
\end{equation*}%
The medical knowledge may be encoded in the form of a look-up tables.
Look-up tables are the computer representation of the notion of mapping in
mathematics. Suppose our medical experts have provided us with following
knowledge:%
\begin{equation*}
\begin{tabular}{llllll}
$u\left( \text{burning in stomach}\right) $ & $=$ & $\text{acidity,}$ & $%
u\left( \text{headache}\right) $ & $=$ & $\text{blood pressure,}$ \\ 
$u\left( \text{sleeplessness}\right) $ & $=$ & $\text{acidity,}$ & $u\left( 
\text{backbone pain}\right) $ & $=$ & wrong posture, \\ 
$u\left( \text{semi-conscious sleep}\right) $ & $=$ & $\text{fatigue,}$ & $%
u\left( \text{joint pain}\right) $ & $=$ & $\text{acidity}$, \\ 
$u\left( \text{depression}\right) $ & $=$ & $\text{low energy level,}$ & $%
u\left( \text{anxiety}\right) $ & $=$ & low energy level,%
\end{tabular}%
\end{equation*}%
and%
\begin{equation*}
\begin{tabular}{ll}
$p\left( \text{high importance}\right) =$ in$\text{frequent high potency,}$
& $p\left( \text{medium importance}\right) =\text{frequent low potency .}$%
\end{tabular}%
\end{equation*}%
For the sake of ease in mathematical manipulation we denote the symptoms and
gradations by symbols as follows:%
\begin{equation*}
\begin{tabular}{lll}
$b$ & $=$ & burning in stomach \\ 
$h$ & $=$ & headache \\ 
$s$ & $=$ & sleeplessness \\ 
$c$ & $=$ & semi-conscious sleep \\ 
$j$ & $=$ & joint pain \\ 
$p$ & $=$ & backbone pain \\ 
$d$ & $=$ & depression \\ 
$a$ & $=$ & anxiety%
\end{tabular}%
\text{ , }%
\begin{tabular}{lll}
$e_{1}$ & $=$ & high importance \\ 
$e_{2}$ & $=$ & medium importance \\ 
$e_{3}$ & $=$ & low importance%
\end{tabular}%
\text{ ~}
\end{equation*}%
and%
\begin{equation*}
\begin{tabular}{lll}
$\alpha $ & $=$ & acidity \\ 
$\beta $ & $=$ & blood pressure \\ 
$\gamma $ & $=$ & fatigue \\ 
$\delta $ & $=$ & wrong posture \\ 
$\lambda $ & $=$ & depression \\ 
$\mu $ & $=$ & mood disorder%
\end{tabular}%
\text{ , }%
\begin{tabular}{lll}
$e_{1}^{\prime }$ & $=$ & infrequent high potency \\ 
$e_{2}^{\prime }$ & $=$ & frequent low potency%
\end{tabular}%
.
\end{equation*}

Thus we have two soft classes $\left( X,E\right) $ and $\left( Y,E^{\prime
}\right) $ with $X=\left\{ b,h,s,c,j,p,d,a\right\} $, $E=\left\{
e_{1},e_{2},e_{3}\right\} $ and $Y=\left\{ \alpha ,\beta ,\gamma ,\delta
,\lambda ,\mu \right\} ,$ $E^{\prime }=\left\{ e_{1}^{\prime },e_{2}^{\prime
}\right\} .$ $\left( X,E\right) $ is the soft class of symptoms and their
importance for the patient, and $\left( Y,E^{\prime }\right) $ represents
causes and medical preference for treatment. The soft set of patient's
narration may be given as:%
\begin{equation*}
\left( F,A\right) =\left\{ e_{1}=\left\{ b,h,s\right\} ,e_{2}=\left\{
c\right\} ,e_{3}=\left\{ j,p,d,a\right\} \right\} .
\end{equation*}%
As a first task of the medical expert system, stored medical knowledge is to
be applied on the given case. This knowledge, in the language of computer
programming, is given as look-up tables.\ Mappings $u:X\rightarrow Y$ and $%
p:E\rightarrow E^{\prime }$ are defined as:%
\begin{eqnarray*}
u\left( b\right) &=&\alpha ,~u\left( h\right) =\beta ,~u\left( s\right)
=\alpha ,~u\left( c\right) =\gamma ,~u\left( j\right) =\alpha ,~u\left(
p\right) =\delta ,~u\left( d\right) =\mu ,~u\left( a\right) =\mu , \\
p\left( e_{1}\right) &=&e_{1}^{\prime },~p\left( e_{2}\right) =e_{2}^{\prime
}.
\end{eqnarray*}%
Calculations give:%
\begin{eqnarray*}
f\left( F,A\right) &=&\left\{ e_{1}^{\prime }=\left\{ \alpha ,\beta \right\}
,e_{2}^{\prime }=\left\{ \gamma \right\} \right\} \\
&=&\left\{ \text{infrequent high potency}=\left\{ \text{acidity, blood
pressure}\right\} ,\text{ frequent low potency }=\left\{ \text{fatigue}%
\right\} \right\} .
\end{eqnarray*}

\begin{conclusion}
A soft set, being a collection of information granules, is the mathematical
formulation of approximate reasoning about information systems. In this
paper, we define the notion of mapping on soft classes. Several properties
of soft images and soft inverse images have been established and supported
by examples and counterexamples. \ Finally, these notions have been applied
to the problem of medical diagnosis. It is hoped that these notions will be
useful for the researchers to further promote and advance this research in
Soft Set Theory.
\end{conclusion}

\end{document}